\documentclass[12pt, a4paper]{article}
\usepackage{}
\usepackage{mathrsfs}
\usepackage{mathrsfs}
\usepackage{amsfonts}
\usepackage{amsfonts}
\usepackage{amsfonts}
\usepackage{mathrsfs}
\usepackage{mathrsfs}
\usepackage{mathrsfs}
\usepackage{mathrsfs}
\usepackage{mathrsfs}
\usepackage{mathrsfs}
\usepackage{mathrsfs}
\usepackage{mathrsfs}
\usepackage{mathrsfs}
\usepackage{mathrsfs}
\usepackage{mathrsfs}
\usepackage{amsfonts}
\usepackage{mathbbold}
\usepackage{amsfonts}
\usepackage{amsfonts}
\usepackage{bbm}
\usepackage{amsfonts}
\usepackage{mathrsfs}
\usepackage{mathrsfs}
\usepackage{mathrsfs}
\usepackage{mathrsfs}
\usepackage{mathrsfs}
\usepackage{mathrsfs}
\usepackage{mathrsfs}
\usepackage{mathrsfs}
\usepackage{mathrsfs}
\usepackage{mathrsfs}
\usepackage{mathrsfs}
\usepackage{mathrsfs}
\usepackage{mathrsfs}
\usepackage{mathrsfs}
\usepackage{mathrsfs}
\usepackage{amsfonts}
\usepackage{mathbbold}
\usepackage{mathrsfs}

\usepackage{latexsym}
\usepackage{graphicx}
\usepackage{amssymb}

\newtheorem{theorem}{Theorem}[section]
\newtheorem{lemma}[theorem]{Lemma}

\title{{\Large \bf  Maxima of the signless Laplacian spectral radius for planar graphs\thanks{Supported by NSFC
(Nos. 11271315, 11171290, 11201417) and Jiangsu Qing Lan Project (2014A).}}}
\author{Guanglong Yu$^a$\thanks{E-mail addresses:
yglong01@163.com.}
~ Jianyong Wang$^{b}$  ~ Shu-Guang Guo$^{a}$ ~
\\ ~ \\
{\footnotesize $^a$Department of Mathematics, Yancheng Teachers
University,}\\ {\footnotesize  Yancheng, 224002, Jiangsu, P.R. China}\\
{\footnotesize $^b$Department of Mathematics, Linyi University,}\\ {\footnotesize  Linyi, 276005, Shandong, P.R. China}}
\date{}

\begin{document}
\maketitle

\begin{abstract}
The signless Laplacian spectral radius of a graph is the largest eigenvalue of its signless Laplacian. In this paper, we prove that the graph $K_{2}\nabla P_{n-2}$ has the maximal signless Laplacian spectral radius among all planar graphs of order $n\geq 456$.


\end{abstract}
\baselineskip 18.6pt

\section{Introduction}

\ \ \ \ The (adjacency) spectral radius of a graph is used in many fields, including chemistry, physics and computer science \cite{CDHS}, \cite{DSKS}, \cite{WCWF}. It arises a broad research now. Spectral radius of planar graphs
is of great interest in harmonic analysis, and bounds on it can be used in the
design and analysis of certain Monte Carlo algorithms \cite{MWW}. Spectral radius of planar graphs have applications not only in the theory of
algorithms but also in theoretical physics. Boots and Royle investigated the spectral radius of planar graphs
motivated by an application in geography networks \cite{BR}. This makes the research about the spectral radius of planar graphs interesting and dynamic, and many interesting results have emerged (see \cite{MWW}, \cite{CAV}, \cite{YH1, YH2}, \cite{SRJW}, for example). Very recently the signless Laplacian has attracted the
attention of researchers. Some results on the signless Laplacian
spectrum have been reported since 2005 and a new
spectral theory called the $Q$-theory is being
developed by many researchers \cite{D.P.S}-\cite{DCSKS}, \cite{FYIS}. A nature question is that how about the application of the $Q$-theory. By computer investigations of graphs \cite{D.H}, E. van Dam and W. Haemers found
that the signless Laplacian spectrum performs better than the
adjacency spectrum or Laplacian spectrum in distinguishing non-isomorphic graphs. In  computer science,  signless Laplacian spectral radius can also be used to study the properties of a network  \cite{DSKS}, \cite{ZJ1, ZJ2}. Schwenk and Wilson initiated the study of the eigenvalues of planar graphs \cite{SRJW}. In \cite{CAV}, D. Cao and A. Vince conjectured that $K_{2}\nabla P_{n-2}$ has the maximum
spectral radius among all planar graphs of order $n$, where
$\nabla$ denotes the $join$ of two graphs obtained from the union of these two graphs by
joining each vertex of the first graph to each vertex of the second
graph. The conjuncture is still open now. With the development of the $Q$-theory, a natural question is that what about the maximum signless Laplacian spectral radius of planar graphs. By some comparisons in \cite{FYIS}, it seems plausible that $K_{2}\nabla P_{n-2}$ also has the
maximal signless Laplacian spectral radius among planar graphs. In this paper, we confirm that among planar graphs with order $n\geq 456$, $K_{2}\nabla P_{n-2}$ has the
maximal signless Laplacian spectral radius.

The layout of this paper is as follows. Section 2
gives some notations and some working lemmas. In section 3,  our results are presented.

\section{Preliminary}

\ \ \ \ \ All graphs considered in this paper are undirected and
simple, i.e. no loops or multiple edges are allowed. Denote by $G=G[V(G)$, $E(G)]$ a graph with vertex set
$V(G)$ and edge set $E(G)$. $|V(G)|= n$ is the order. $m(G)=|E(G)|$ is edge number. Recall that given a graph $G$, $Q(G)= D(G) + A(G)$ is called the $signless$ $Laplacian$ $matrix$ of $G$, where $D(G)= \mathrm{diag}(d_{1}, d_{2},
\ldots, d_{n})$ with $d_{i}= d_{G}(v_{i})$ being the degree of
vertex $v_{i}$ $(1\leq i\leq n)$, and $A(G)$ is the adjacency matrix of $G$. The signless Laplacian spectral radius of $G$ is the largest eigenvalue $q(G)$ of its signless Laplacian $Q(G)$. From spectral graph theory, we know that if graph $G$ is connected, then there is a positive real eigenvector corresponding to $q(G)$, and the unit positive eigenvector corresponding to $q(G)$ is always called $Perron$ $eigenvector$. For a graph $G$ of order $n$, if $X=(x_{1}, x_{2}, \ldots, x_{n})^{T} \in R^{n}$ is a positive eigenvector corresponding to $q(G)$ satisfying that $\displaystyle \sum^{n}_{i=1}x_{i}=1$, then $X$ is called a $standard$ $eigenvector$ corresponding to $q(G)$.

Denote by $K_{n}$, $C_{n}$, $P_{n}$ a $complete$ $graph$, a $cycle$ and a $path$ of order $n$ respectively.
For a graph $G$, we let $V(G)$, $E(G)$ denote the vertex set and the edge set respectively. If there is no ambiguity, we use $d(v)$ instead of $d_{G}(v)$. We use $\delta$ or $\delta(G)$ to denote the minimum vertex degree of a graph, use $\Delta$ or $\Delta(G)$ to denote the largest vertex degree of a graph, and use $\Delta^{'}$ or $\Delta^{'}(G)$ to denote the second largest vertex degree in a graph. In a graph, the notation $v_{i}\sim v_{j}$ denotes that vertex $v_{i}$ is adjacent to
$v_{j}$. Denote by $K_{s,t}$ a complete bipartite graph with one part of size $s$ and another part of size $t$. In a graph $G$ of order $n\geq 4$, for a vertex $u\in V(G)$, let $N_{G}(u)$ denote the neighbor set of $u$, and let $N_{G}[u]=\{u\}\cup N_{G}(u)$. $G(u)=G[N_{G}[u]]$, $G^{\circ}(u)=G[N_{G}(u)]$ denote the subgraphs induced by $N_{G}[u]$, $N_{G}(u)$ respectively.

The reader is referred to \cite{BM, FH} for the
facts about planar and outer-planar graphs. A graph which can be drawn in the
plane in such a way that edges meet only at points corresponding to their common
ends is called a planar graph, and such a drawing is called a planar embedding
of the graph. A simple planar graph is (edge) $maximal$ if no edge can be added to the graph without violating planarity. In the planar embedding of a maximal planar graph $G$ of order $n\geq 3$, each face is triangle. For a planar graph $G$, we have $m(G)\leq 3n-6$ with equality if and only if it is maximal. In a maximal planar graph $G$ of order $n\geq 4$, $\delta(G)\geq 3$. A graph $G$ is $outer$-$planar$ if it has a planar embedding, called $standard$ $embedding$, in which all vertices lie on the boundary of its outer face. A simple outer-planar graph is (edge) $maximal$ if no edge can be added to the graph without violating outer-planarity. In the standard embedding of a maximal outer-planar graph $G$ of order $n\geq 3$, the boundary of the outer face is a Hamiltonian
cycle (a cycle contains all vertices) of $G$, and each of the other faces is triangle. For an outer-planar graph $G$, we have $m(G)\leq 2n-3$ with equality if and only if it is maximal. In a maximal planar graph $G$ of order $n\geq 4$, for a vertex $u\in V(G)$, $G^{\circ}(u)$ is an outer-planar graph, and $G(u)= u\nabla G^{\circ}(u)$. From a nonmaximal planar graph $G$, by adding edges to $G$, a maximal planar graph $G^{'}$ can be obtained. From spectral graph theory, for a graph $G$, it is known that $q(G+e)> q(G)$ if $e\notin E(G)$. Consequently, when we consider the maxima of the signless Laplacian spectral radius among planar graphs, it suffices to consider the maximal planar graphs directly.

Next we introduce some working lemmas.

\begin{lemma}{\bf \cite{SHH}}\label{le2,1} 
Let $u$ be a vertex of a maximal outer-planar graph on $n\geq 2$ vertices.
Then $\displaystyle\sum_{v\sim u}d(v)\leq n+3d(u)-4$.
\end{lemma}

\begin{lemma}{\bf \cite{R.M}}\label{le2,3} 
Let $G$ be a graph. Then $\displaystyle q(G)\leq \max_{u\in V(G)}\{d_{G}(u)+\frac{1}{d_{G}(u)}\sum_{v\sim u}d_{G}(v)\}$.
\end{lemma}

\begin{lemma}{\bf \cite{DSK}}\label{le2,4} 
Let $G$ be a connected graph containing at least one edge. Then $\displaystyle q(G)\geq \Delta+1$ with equality if and only if $G\cong K_{1, n-1}$.
\end{lemma}

\section{Main results}

\begin{lemma}\label{le3,1} 
Let $G$ be a maximal planar graph with order $n\geq 6$ and largest degree $\Delta(G)$. Then
$\displaystyle q(G)\leq \max_{u\in V(G)}\{d_{G}(u)+2+\frac{3n-9}{d_{G}(u)}\}$.
Moreover,\\
$\mathrm{(i)}$ if $\Delta(G)=n-1$, then $q(G)\leq n+4-\frac{6}{n-1}$;\\
$\mathrm{(ii)}$ if $\Delta(G)=n-2$, then $q(G)\leq n+3-\frac{3}{n-2}$;\\
$\mathrm{(iii)}$ if $\Delta(G)\leq n-3$, then $q(G)\leq n+2$.

\end{lemma}

\begin{proof}
For any vertex $u\in V(G)$, let $N_{G}(u)=\{v_{1}, v_{2}, \ldots, v_{t}\}$, $V_{1}=V(G)\backslash N_{G}[u]$. For $1\leq i\leq t$, let $\alpha_{i}=d_{G^{\circ}(u)}(v_{i})$. Note that $\displaystyle m(G^{\circ}(u))=|E(G^{\circ}(u))|=\frac{1}{2}\sum^{t}_{i=1}\alpha_{i}$. Between $N_{G}(u)$ and $V_{1}$, there are $\displaystyle 3n-6-\frac{1}{2}\sum^{t}_{i=1}\alpha_{i}-d_{G}(u)$ edges. Consequently,
$$\displaystyle\sum_{v\sim u}d_{G}(v)= d_{G}(u)+[3n-6-\frac{1}{2}\sum^{t}_{i=1}\alpha_{i}-d_{G}(u)]+\sum^{t}_{i=1}\alpha_{i}=
3n-6+\frac{1}{2}\sum^{t}_{i=1}\alpha_{i}.$$ From the narration in Section 2, we know that $G^{\circ}(u)$ is an outer-planar graph. Then $m(G^{\circ}(u))\leq 2d_{G}(u)-3$. As a result,
$\displaystyle\sum_{v\sim u}d_{G}(v)\leq 3n-9+2d_{G}(u)$,
and
$$d_{G}(u)+\frac{1}{d_{G}(u)}\sum_{v\sim u}d_{G}(v)\leq d_{G}(u)+2+\frac{3n-9}{d_{G}(u)}$$
By Lemma \ref{le2,3}, $\displaystyle q(G)\leq \max_{u\in V(G)}\{d_{G}(u)+2+\frac{3n-9}{d_{G}(u)}\}$.

Let $f(x)=x+2+\frac{3n-9}{x}$. It can be checked that $f(x)$ is convex.
Note that $3\leq d_{G}(u)\leq \Delta$. Then $$d_{G}(u)+\frac{1}{d_{G}(u)}\sum_{v\sim u}d_{G}(v)\leq\max\{5+\frac{3n-9}{3}, \Delta+2+\frac{3n-9}{\Delta}\}.$$ By Lemma \ref{le2,3}, then $\mathrm{(i)}$-$\mathrm{(iii)}$ follow. This completes the proof.
\ \ \ \ \ $\Box$
\end{proof}

Let $\mathcal {H}=K_{2}\nabla P_{n-2}$ for $n\geq 4$ (see Fig. 3.1) and let $\mathcal {H}=K_{n}$ for $n=1, 2, 3$.

\setlength{\unitlength}{0.6pt}
\begin{center}
\begin{picture}(488,203)
\put(472,188){\circle*{4}}
\put(472,23){\circle*{4}}
\qbezier(472,188)(472,106)(472,23)
\put(24,102){\circle*{4}}
\put(135,102){\circle*{4}}
\qbezier(24,102)(79,102)(135,102)
\put(185,101){\circle*{4}}
\put(200,101){\circle*{4}}
\put(214,101){\circle*{4}}
\put(82,102){\circle*{4}}
\put(266,102){\circle*{4}}
\put(373,102){\circle*{4}}
\qbezier(266,102)(319,102)(373,102)
\put(316,102){\circle*{4}}
\qbezier(373,102)(422,63)(472,23)
\qbezier(373,102)(422,145)(472,188)
\qbezier(316,102)(394,63)(472,23)
\qbezier(316,102)(394,145)(472,188)
\qbezier(266,102)(369,63)(472,23)
\qbezier(266,102)(369,145)(472,188)
\qbezier(135,102)(303,63)(472,23)
\qbezier(135,102)(303,145)(472,188)
\qbezier(82,102)(277,63)(472,23)
\qbezier(82,102)(277,145)(472,188)
\qbezier(24,102)(248,63)(472,23)
\qbezier(24,102)(248,145)(472,188)
\qbezier(266,102)(249,102)(232,102)
\qbezier(135,102)(151,102)(168,102)
\put(473,190){$v_{1}$}
\put(475,19){$v_{2}$}
\put(380,98){$v_{3}$}
\put(302,108){$v_{4}$}
\put(255,108){$v_{5}$}
\put(3,104){$v_{n}$}
\put(202,-9){Fig 3.1. $\mathcal {H}$}
\end{picture}
\end{center}

\begin{lemma}\label{le3,2} 
Suppose that the order $n$ of $\mathcal {H}$ is at least $5$ (see Fig. 3.1). Then $q(\mathcal {H})> n+2$.
\end{lemma}

\begin{proof}
Let $X=(x_1, x_2, \ldots, x_{n})^T \in R^{n}$ be a standard eigenvector corresponding to $q(\mathcal {H})$, where $x_{i}$ corresponds to vertex $v_{i}$ and $\displaystyle \sum^{n}_{i=1}x_{i}=1$. By symmetry, $x_{1}=x_{2}$, $x_{3}=x_{n}$. By Lemma
\ref{le2,4}, $q(G)\geq n$.

Note that $$\displaystyle q(G)x_1=(n-1)x_1+x_2+\sum^{n}_{i=3}x_{i}=(n-2)x_1+x_1+x_2+\sum^{n}_{i=3}x_{i}=(n-2)x_1+1.$$ Then $(q(G)-n+2)x_1=1$, and $$x_1=\frac{1}{q(G)-n+2}. \hspace{3cm} (1)$$

Note that $$\displaystyle q(G)\sum^{n}_{i=3}x_{i}=3x_{3}+3x_{n}+4\sum^{n-1}_{i=4}x_{i}+x_{3}+x_{n}+2\sum^{n-1}_{i=4}x_{i}+(n-2)x_1+(n-2)x_2$$
$$\, =3x_{3}+3x_{n}+4\sum^{n-1}_{i=4}x_{i}+x_{3}+x_{n}+2\sum^{n-1}_{i=4}x_{i}+2(n-2)x_1$$
$$\displaystyle =6\sum^{n}_{i=3}x_{n}+2(n-2)x_1-2(x_{3}+x_{n}).\ \, \hspace{3.1cm}$$
Then $$\displaystyle (q(G)-6)\sum^{n}_{i=3}x_{i}=2(n-2)x_1-2(x_{3}+x_{n})=2(n-2)x_1-4x_{3},$$
 $$\displaystyle \sum^{n}_{i=3}x_{i}=\frac{2(n-2)x_1-4x_{3}}{q(G)-6},\hspace{2.8cm}$$
 and $$\displaystyle 1=\sum^{n}_{i=1}x_{i}=\frac{2(n-2)x_1-4x_{3}}{q(G)-6}+x_1+x_2$$$$=\frac{2(n-2)x_1-4x_{3}}{q(G)-6}+\frac{2}{q(G)-n+2}\ \, $$
$$\displaystyle =\frac{\frac{2(n-2)}{q(G)-n+2}-4x_{3}}{q(G)-6}+\frac{2}{q(G)-n+2}.\ \ \, \hspace{0.1cm}$$
Then $$\hspace{1.8cm} x_{3}=\frac{2(n-2)-(q(G)-n)(q(G)-6)}{4(q(G)-n+2)}. \hspace{2cm} (2)$$
Note that $q(G)x_3=3x_3+x_1+x_2+x_4$. Then $$\displaystyle q(G)x_1-q(G)x_3=(n-2)x_1-2x_3+\sum^{n}_{i=5}x_{i}$$
$$\hspace{4.3cm}\, \displaystyle =(n-4)x_1+2x_1-2x_3+\sum^{n}_{i=5}x_{i},$$
and then $$\displaystyle (q(G)-2)(x_1-x_3)=(n-4)x_1+\sum^{n}_{i=5}x_{i}.$$
This implies that $x_1>x_3$. Combining with (1) and (2), we get
$$\frac{2(n-2)-(q(G)-n)(q(G)-6)}{4(q(G)-n+2)}< \frac{1}{q(G)-n+2}. \hspace{1cm} (3)$$
Simplifying (3), we get $q^{2}(G)-(6+n)q(G)+4n+8>0$. Then
$$\displaystyle q(G)>\frac{6+n+\sqrt{(6+n)^{2}-16(n+2)}}{2}=n+2.$$
This
completes the proof. \ \ \ \ \ $\Box$
\end{proof}

By the narration in Section 2 and Lemmas \ref{le3,1}, \ref{le3,2}, to consider the maxima of the signless Laplacian spectral radius among planar graphs of order $n\geq 5$, it suffices to consider only the graphs with $\Delta=n-1$ and $\Delta=n-2$.

\begin{lemma}{\bf \cite{H.M}}\label{le3,3,0} 
Let $A$ be an irreducible nonnegative square real matrix of order $n$ and $X=(x_{1}$, $x_{2}$, $x_{3}$, $\ldots$, $x_{n})^{T}$ be a real vector. For $1\leq i\leq n$, let $A_{i}$ be the $ith$ row of $A$. If for any $1\leq i\leq n$, $A_{i}X\leq rx_{i}$, we say that $AX\leq rX$.
\end{lemma}

\begin{lemma}{\bf \cite{GYSH}}\label{le3,3} 
Let $A$ be an irreducible nonnegative square real matrix of order $n$ and spectral radius $\rho$. If there exists a nonnegative real vector $y\neq 0$ and a real coefficient polynomial function $f$ such that $f(A)y\leq ry$ $(r\in {\bf R})$, then $f(\rho)\leq r$.
\end{lemma}

\begin{lemma}\label{le3,4} 
Let $G$ be a maximal planar graph with order $n \geq 115$ and $d_{G}(v_{1})=\Delta(G) = n -2$. In $G$, there are exactly $1\leq k\leq 12$ vertices $v_{2}$, $v_{3}$, $\ldots$, $v_{k+1}$ such that $\frac{n}{6}+1\leq d_{G}(v_{i})\leq n-61,$ and for $k+2\leq i\leq n$, $d_{G}(v_{i})<\frac{n}{6}+1$. Then $q(G)\leq n-2$.
\end{lemma}

\begin{proof}
Note that in $G$, $d_{G}(v_{i})\geq 3$ for $1\leq i \leq n$, $\displaystyle \sum^{n}_{i=1}d_{G}(v_{i})=2(3n-6)$ and $$k(\frac{n}{6}+1)+n-2+3(n-k-1)=(\frac{k}{6}+4)n-2k-5 \hspace{1.3cm}  (\mathrm{if}\ k\geq 13)\hspace{1.3cm} \ \ $$$$\hspace{1.8cm} \ \ \geq 6n+\frac{n}{6}-31>6n-12.$$ Hence, $k\leq 12$.
Let $X=(x_{1}$, $x_{2}$, $x_{3}$, $\ldots$, $x_{n})^{T}$ be a positive vector satisfying that $x_{i}$ corresponds to vertex $v_{i}$ and
 $$x_{i}=\left \{\begin{array}{ll}
 1,\ &
 i=1;\\
\\ \frac{1}{k},\ &
 2\leq i\leq k+1;\\
\\ \frac{3}{n-k-1},\ & k+2\leq i\leq n.\end{array}\right.$$

For $x_{1}$, $$\frac{(n-2)x_{1}+\sum_{v_{j}\sim v_{1}} x_{j}}{x_{1}}\leq n-2+1+\frac{3(n-k-2)}{n-k-1}<n+2.$$
For $x_{i}$ ($k+2\leq i\leq n$), when $d(v_{i})\geq k+1$,
$$\frac{d_{G}(v_{i})x_{i}+\sum_{v_{j}\sim v_{i}} x_{j}}{x_{i}}\leq d_{G}(v_{i})+\frac{\sum^{k+1}_{j=1} x_{j}+\frac{3(d_{G}(v_{i})-k-1)}{n-k-1}}{\frac{3}{n-k-1}}\hspace{0.1cm}$$$$\hspace{2.2cm}\ \leq d_{G}(v_{i})+\frac{2+\frac{3(d_{G}(v_{i})-k-1)}{n-k-1}}{\frac{3}{n-k-1}}$$
$$\hspace{5cm}\leq d_{G}(v_{i})+\frac{2}{3}(n-k-1)+d_{G}(v_{i})-k-1\ $$
$$\hspace{3.6cm}=2d_{G}(v_{i})+\frac{2}{3}n-\frac{5}{3}k-\frac{5}{3}\leq n+2;$$
when $d_{G}(v_{i})\leq k$, $$\frac{d_{G}(v_{i})x_{i}+\sum_{v_{j}\sim v_{i}} x_{j}}{x_{i}}\leq d_{G}(v_{i})+\frac{ \sum^{d_{G}(v_{i})}_{j=1}x_{j}}{\frac{3}{n-k-1}}\hspace{3.2cm}$$$$\hspace{1cm}\leq d_{G}(v_{i})+\frac{1+\frac{d_{G}(v_{i})-1}{k}}{\frac{3}{n-k-1}}\ $$
$$\hspace{5.3cm}=d_{G}(v_{i})+\frac{n-k-1}{3}+\frac{(d_{G}(v_{i})-1)(n-k-1)}{3k}$$
$$\hspace{2cm}<k+\frac{n-k-1}{3}+\frac{n-k-1}{3}$$$$\hspace{1.3cm}=\frac{2(n-1)+k}{3}<n+2.\ $$
For $x_{i}$ ($2\leq i\leq k+1$), noting that $d_{G}(v_{i})\geq\frac{n}{6}+1>k$, we get
$$\ \frac{d_{G}(v_{i})x_{i}+\sum_{v_{j}\sim v_{i}} x_{j}}{x_{i}} =\frac{(d_{G}(v_{i})-1)x_{i}+x_{i}+\sum_{v_{j}\sim v_{i}} x_{j}}{x_{i}}\hspace{4.5cm}$$$$\leq d_{G}(v_{i})-1+\frac{\sum^{k+1}_{j=1} x_{j}+\frac{3(d_{G}(v_{i})-k)}{n-k-1}}{\frac{1}{k}}$$
$$ = d_{G}(v_{i})-1+\frac{2+\frac{3(d_{G}(v_{i})-k)}{n-k-1}}{\frac{1}{k}}\hspace{1.2cm}$$
$$\hspace{3.5cm}=(1+\frac{3k}{n-k-1})d_{G}(v_{i})-\frac{3k^{2}}{n-k-1}+2k-1.\hspace{1cm} (4)$$
Let $\displaystyle f(k)=(1+\frac{3k}{n-k-1})d_{G}(v_{i})-\frac{3k^{2}}{n-k-1}+2k-1$. Taking derivation of $f(k)$ with respect to $k$,
we get $\displaystyle f^{'}(k)=\frac{(n-k-1)(2n-2+3d_{G}(v_{i})-8k)+3kd_{G}(v_{i})-3k^{2}}{(n-k-1)^{2}}$. Note that
$n\geq 115$. As a result, if $d_{G}(v_{i})\geq k$, then $f^{'}(k)> 0$. This implies that $f(k)$ is monotone increasing with respect to $k$. Note that $d_{G}(v_{i})\leq n-61$. Then $(4)<n+2.$

Then $Q(G)X\leq(n+2)X$. By Lemma \ref{le3,3}, $q(G)\leq n+2$.
This completes the proof.
\ \ \ \ \ $\Box$
\end{proof}

\begin{lemma}\label{le3,5} 
Let $G$ be a maximal planar graph with order $n \geq 380$, $d_{G}(v_{1})=\Delta(G) = n -2$, and $\Delta^{'}(G)\geq n-62$. Then $q(G)\leq n+2$.
\end{lemma}

\begin{proof}
Suppose $d_{G}(v_{2})=\Delta^{'}(G)$. Let $X=(x_{1}$, $x_{2}$, $x_{3}$, $\ldots$, $x_{n})^{T}$ be a positive vector satisfying that $x_{i}$ corresponds to vertex $v_{i}$ and
 $$x_{i}=\left \{\begin{array}{ll}
 1,\ &
 i=1;\\
\\ 1,\ &
 i=2;\\
\\ \frac{3}{n-2},\ & 3\leq i\leq n.\end{array}\right.$$
For $v_{1}$, $$\frac{(n-2)x_{1}+\sum_{v_{j}\sim v_{1}} x_{j}}{x_{1}}\leq n-2+1+\frac{3(n-3)}{n-2}<n+2.$$

Next, there are two cases to consider.

{\bf Case 1} $v_{2}\in N_{G}(v_{1})$. Suppose $N_{G}(v_{1})=\{v_{2}$, $v_{3}$, $\ldots$, $v_{n-2}$, $v_{n-1}\}$.
Suppose $v_{1}$ is in the outer face of $G^{\circ}_{v_{1}}$. Then $v_{n}$ is in one of the inner faces of $G^{\circ}_{v_{1}}$ (see Fig. 3.2).

For $v_{2}$, $$\frac{d_{G}(v_{2})x_{2}+\sum_{v_{j}\sim v_{2}} x_{j}}{x_{2}}\leq d_{G}(v_{2})+1+\frac{3(d_{G}(v_{2})-1)}{n-2}\ \hspace{0.4cm}$$
$$\hspace{4cm} =d_{G}(v_{2})+1+\frac{3d_{G}(v_{2})}{n-2}-\frac{3}{n-2}$$$$\hspace{5.3cm} \leq n+2-\frac{3}{n-2}. \hspace{1cm} (d_{G}(v_{2})\leq n-2)$$
Denote by $C_{v_{1}}=v_{2}v_{3}\cdots v_{n-2}v_{n-1}v_{2}$ the Hamiltonian
cycle in $G^{\circ}_{v_{1}}$. Suppose that $v_{i}s$ ($2\leq i\leq n-1$) are distributed along clockwise direction on $C_{v_{1}}$ and suppose $N_{G^{\circ}_{v_{1}}}(v_{2})=\{v_{2_{1}}$, $v_{2_{2}}$, $\ldots$, $v_{2_{t}}\}$, where for $1\leq i\leq t-1$, $2_{i}<2_{i+1}$, $v_{2_{1}}=v_{3}$, $v_{2_{t}}=v_{n-1}$ (see Fig. 3.2).
For $1\leq j\leq t$, suppose there are $l_{j-1}$ vertices between $v_{2_{j-1}}$ and $v_{2_{j}}$ along clockwise direction on $C_{v_{1}}$, where if $j=1$, we let $v_{2_{0}}=v_{2}$. Along clockwise direction on $C_{v_{1}}$, suppose there are $l_{t}$ vertices between $v_{2_{t}}$ and $v_{2}$.

\setlength{\unitlength}{0.6pt}
\begin{center}
\begin{picture}(532,224)
\put(166,198){\circle*{4}}
\put(156,45){\circle*{4}}
\put(122,41){\circle*{4}}
\put(217,149){\circle*{4}}
\qbezier(122,41)(169,95)(217,149)
\put(93,44){\circle*{4}}
\put(29,103){\circle*{4}}
\qbezier(122,41)(75,72)(29,103)
\put(44,171){\circle*{4}}
\qbezier(122,41)(83,106)(44,171)
\put(145,204){\circle*{4}}
\qbezier(122,41)(133,123)(145,204)
\put(116,26){$v_{2}$}
\put(81,31){$v_{3}$}
\put(7,180){$v_{2_{j-1}}$}
\put(221,150){$v_{2_{j+1}}$}
\put(151,31){$v_{n-1}$}
\put(123,206){\circle*{4}}
\put(435,204){\circle*{4}}
\put(99,203){\circle*{4}}
\put(186,187){\circle*{4}}
\put(187,72){\circle*{4}}
\put(177,64){\circle*{4}}
\put(167,58){\circle*{4}}
\put(96,130){\circle*{4}}
\put(75,62){\circle*{4}}
\put(65,69){\circle*{4}}
\put(56,77){\circle*{4}}
\put(41,146){\circle*{4}}
\put(40,132){\circle*{4}}
\put(39,118){\circle*{4}}
\put(88,190){\circle*{4}}
\put(76,183){\circle*{4}}
\put(62,175){\circle*{4}}
\put(203,156){\circle*{4}}
\put(196,165){\circle*{4}}
\put(190,174){\circle*{4}}
\put(102,125){$v_{n}$}
\qbezier[17](96,130)(120,167)(145,204)
\put(60,-9){Fig. 3.2. $d_{G}(v_{2_{j}})$}
\qbezier(337,123)(337,156)(365,180)\qbezier(365,180)(393,204)(434,204)\qbezier(434,204)(474,204)(502,180)
\qbezier(502,180)(531,156)(531,123)\qbezier(531,123)(531,89)(502,65)\qbezier(502,65)(474,41)(434,41)
\qbezier(434,42)(393,42)(365,65)\qbezier(365,65)(336,89)(337,123)
\put(424,43){\circle*{4}}
\put(337,115){\circle*{4}}
\qbezier(424,43)(380,79)(337,115)
\put(525,96){\circle*{4}}
\qbezier(424,43)(474,70)(525,96)
\put(455,43){\circle*{4}}
\put(391,50){\circle*{4}}
\put(371,73){\circle*{4}}
\put(364,79){\circle*{4}}
\put(359,86){\circle*{4}}
\put(497,72){\circle*{4}}
\put(487,66){\circle*{4}}
\put(478,60){\circle*{4}}
\put(419,29){$v_{2}$}
\put(380,35){$v_{3}$}
\put(309,115){$v_{2_{s}}$}
\put(531,91){$v_{2_{s+1}}$}
\put(452,30){$v_{n-1}$}
\put(430,214){$v_{f}$}
\put(368,182){\circle*{4}}
\put(477,196){\circle*{4}}
\put(394,197){\circle*{4}}
\put(501,181){\circle*{4}}
\put(442,118){\circle*{4}}
\qbezier[17](435,204)(438,161)(442,118)
\put(353,138){\circle*{4}}
\put(357,148){\circle*{4}}
\put(363,156){\circle*{4}}
\put(519,135){\circle*{4}}
\put(515,144){\circle*{4}}
\put(509,153){\circle*{4}}
\put(138,214){$v_{2_{j}}$}
\put(373,207){$v_{f-1}$}
\put(476,203){$v_{f+1}$}
\put(418,121){$v_{n}$}
\put(380,-8){Fig. 3.3. $d_{G}(v_{f})$}
\put(213,130){\circle*{4}}
\put(212,120){\circle*{4}}
\put(211,111){\circle*{4}}
\put(217,98){\circle*{4}}
\qbezier(122,41)(169,70)(217,98)
\put(2,102){$v_{2_{1}}$}
\put(222,93){$v_{2_{t}}$}
\qbezier[23](145,204)(86,139)(44,171)
\qbezier[11](99,203)(104,178)(145,204)
\qbezier(26,123)(26,157)(54,181)\qbezier(54,181)(83,206)(124,206)\qbezier(124,206)(164,206)(193,181)
\qbezier(193,181)(222,157)(222,123)\qbezier(222,123)(222,88)(193,64)\qbezier(193,64)(164,40)(124,40)
\qbezier(124,40)(83,40)(54,64)\qbezier(54,64)(26,88)(26,123)
\qbezier[11](145,204)(174,165)(186,187)
\qbezier[21](145,204)(163,126)(217,149)
\qbezier[15](435,204)(478,158)(501,181)
\qbezier[28](435,204)(469,143)(525,96)
\qbezier[15](435,204)(398,162)(368,182)
\qbezier[28](435,204)(410,125)(337,115)
\end{picture}
\end{center}

For each $v_{2_{j}}$ ($1\leq j\leq t$, see Fig. 3.2), noting that $$l_{j-1}+l_{j}\leq n-3-d_{G}(v_{2})$$ and $d_{G}(v_{2})\geq n-62$, then $$d_{G}(v_{2_{j}})\leq l_{j-1}+l_{j}+5\leq n+2-d_{G}(v_{2})\leq 64,$$ and then $$\displaystyle \frac{d_{G}(v_{2_{j}})x_{2_{j}}+\sum_{v_{k}\sim v_{2_{j}}} x_{k}}{x_{2_{j}}}\leq d_{G}(v_{2_{j}})+\frac{2+\frac{3(d_{G}(v_{2_{j}})-2)}{n-2}}{\frac{3}{n-2}}  \hspace{4cm}$$
$$\ \ \ = 2d_{G}(v_{2_{j}})-2+\frac{2}{3}(n-2)$$$$\hspace{3cm}\ \  \leq 126 +\frac{2}{3}(n-2)\leq n+2.  \hspace{1cm} (n\geq 380)$$
For each $v_{f}\in (N_{G}(v_{1})\backslash \{v_{2}$, $v_{2_{1}}$, $v_{2_{2}}$, $\ldots$, $v_{2_{t}}\})$, then along clockwise direction on $C_{v_{1}}$, there exists $0\leq s\leq t$ such that $v_{f}$ is between $v_{2_{s}}$ and $v_{2_{s+1}}$, where $v_{2_{t+1}}=v_{2}$ (see Fig. 3.3). Note that $$l_{s}\leq n-3-d_{G}(v_{2}).$$ Then $d_{G}(v_{f})\leq l_{s}+3\leq n-d_{G}(v_{2})\leq 62$. As a result, $$\frac{d_{G}(v_{f})x_{f}+\sum_{v_{k}\sim v_{f}} x_{k}}{x_{f}}\leq d_{G}(v_{f})+\frac{2+\frac{3(d_{G}(v_{f})-2)}{n-2}}{\frac{3}{n-2}}  \hspace{3.7cm}$$
$$\ \ \ = 2d_{G}(v_{f})-2+\frac{2}{3}(n-2)$$$$\hspace{3cm}\ \  \leq 122 +\frac{2}{3}(n-2)\leq n+2.  \hspace{1cm} (n\geq 380)$$

\setlength{\unitlength}{0.6pt}
\begin{center}
\begin{picture}(516,206)
\qbezier(4,122)(4,157)(30,181)\qbezier(30,181)(58,206)(97,206)\qbezier(97,206)(135,206)(163,181)
\qbezier(163,181)(190,157)(190,122)\qbezier(190,122)(190,86)(163,62)\qbezier(163,62)(135,37)(97,37)
\qbezier(97,38)(58,38)(30,62)\qbezier(30,62)(3,86)(4,122)
\put(88,38){\circle*{4}}
\put(4,125){\circle*{4}}
\qbezier(88,38)(46,82)(4,125)
\put(189,106){\circle*{4}}
\qbezier(88,38)(138,72)(189,106)
\put(24,174){\circle*{4}}
\qbezier(88,38)(56,106)(24,174)
\put(176,166){\circle*{4}}
\qbezier(88,38)(132,102)(176,166)
\put(23,150){\circle*{4}}
\put(91,122){\circle*{4}}
\put(136,45){\circle*{4}}
\put(45,52){\circle*{4}}
\put(52,195){\circle*{4}}
\qbezier[16](91,122)(71,159)(52,195)
\put(155,188){\circle*{4}}
\qbezier[18](91,122)(123,155)(155,188)
\qbezier[16](91,122)(57,148)(24,174)
\qbezier[19](91,122)(133,144)(176,166)
\qbezier[16](91,122)(89,80)(88,38)
\put(23,137){\circle*{4}}
\put(23,124){\circle*{4}}
\put(121,185){\circle*{4}}
\put(103,187){\circle*{4}}
\put(85,188){\circle*{4}}
\put(179,118){\circle*{4}}
\put(177,131){\circle*{4}}
\put(175,143){\circle*{4}}
\put(83,25){$v_{2}$}
\put(34,39){$v_{3}$}
\put(35,73){\circle*{4}}
\put(27,84){\circle*{4}}
\put(20,95){\circle*{4}}
\put(157,68){\circle*{4}}
\put(164,76){\circle*{4}}
\put(172,85){\circle*{4}}
\put(1,182){$v_{2_{z}}$}
\put(179,167){$v_{2_{z+1}}$}
\put(94,112){$v_{n}$}
\put(133,32){$v_{n-1}$}
\put(34,-9){Fig. 3.4. $d_{G}(v_{n})$}
\qbezier(334,121)(334,154)(359,177)\qbezier(359,177)(385,201)(423,201)\qbezier(423,201)(460,201)(486,177)
\qbezier(486,177)(512,154)(512,121)\qbezier(512,121)(512,87)(486,64)\qbezier(486,64)(460,40)(423,40)
\qbezier(423,41)(385,41)(359,64)\qbezier(359,64)(333,87)(334,121)
\put(469,155){\circle*{4}}
\put(449,167){\circle*{4}}
\put(362,80){\circle*{4}}
\put(465,49){\circle*{4}}
\put(389,47){\circle*{4}}
\put(428,40){\circle*{4}}
\put(424,124){\circle*{4}}
\put(341,90){\circle*{4}}
\qbezier(424,124)(382,107)(341,90)
\put(371,71){\circle*{4}}
\put(381,62){\circle*{4}}
\put(349,166){\circle*{4}}
\qbezier(424,124)(386,145)(349,166)
\put(511,106){\circle*{4}}
\qbezier(424,124)(467,115)(511,106)
\put(428,178){\circle*{4}}
\put(492,89){\circle*{4}}
\put(485,79){\circle*{4}}
\put(476,70){\circle*{4}}
\put(418,109){$v_{2}$}
\put(420,27){$v_{3}$}
\put(380,34){$v_{4}$}
\put(463,36){$v_{n}$}
\put(315,87){$v_{2_{1}}$}
\put(324,172){$v_{2_{2}}$}
\put(514,101){$v_{2_{t}}$}
\put(363,-8){Fig. 3.5. $d_{G}(v_{2})$}
\end{picture}
\end{center}

For $v_{n}$, note that $v_{n}$ is in one of the inner faces of $G^{\circ}_{v_{1}}$. Suppose that in $G^{\circ}_{v_{1}}$, $v_{n}$ is in a face $v_{2}v_{2_{z}}v_{2_{z}+1}v_{2_{z}+2}\cdots v_{2_{z+1}}v_{2}$ (see Fig. 3.4). Note that $l_{z}\leq n-3-d_{G}(v_{2})$ and $d_{G}(v_{n})\leq l_{z}+3$. Then $d_{G}(v_{n})\leq n-d_{G}(v_{2})\leq 62$, and $$\frac{d_{G}(v_{n})x_{n}+\sum_{v_{k}\sim v_{n}} x_{k}}{x_{n}}\leq d_{G}(v_{n})+\frac{1+\frac{3(d_{G}(v_{n})-1)}{n-2}}{\frac{3}{n-2}}  \hspace{3.8cm}$$
$$= 2d_{G}(v_{n})-1+\frac{n-2}{3}\ \ $$$$\hspace{3cm}\ \  \leq 123 +\frac{1}{3}(n-2)< n+2  \hspace{1.1cm} (n\geq 380).$$

{\bf Case 2} $v_{2}\notin N_{G}(v_{1})$. Suppose that $v_{1}$ is in the outer face of $G^{\circ}_{v_{1}}$. Then $v_{2}$ is in one of the inner faces of $G^{\circ}_{v_{1}}$. Then $N_{G}(v_{1})=\{v_{3}$, $v_{4}$, $v_{5}$, $\ldots$, $v_{n-1}$, $v_{n}\}$. Suppose that $C_{v_{1}}=v_{3}v_{4}\cdots v_{n-1}v_{n}v_{3}$ is the Hamiltonian
cycle in $G^{\circ}_{v_{1}}$, $v_{i}s$ ($3\leq i\leq n$) are distributed along clockwise direction on $C_{v_{1}}$, and suppose $N_{G^{\circ}_{v_{1}}}(v_{2})=\{v_{2_{1}}$, $v_{2_{2}}$, $\ldots$, $v_{2_{t}}\}$, where for $1\leq i\leq t-1$, $2_{i}<2_{i+1}$ (see Fig. 3.5). For $2\leq j\leq t$, along clockwise direction on $C_{v_{1}}$, suppose there are $l_{j-1}$ vertices between $v_{2_{j-1}}$ and $v_{2_{j}}$. Suppose that there are $l_{t}$ vertices between $v_{2_{t}}$ and $v_{2_{1}}$. 

For each $v_{2_{j}}$ ($2\leq j\leq t$), noting that $l_{j-1}+l_{j}\leq n-2-d_{G}(v_{2})$ and $d_{G}(v_{2})\geq n-62$, then $d_{G}(v_{2_{j}})\leq l_{j-1}+l_{j}+4\leq n+2-d_{G}(v_{2})\leq 64$; for $v_{2_{1}}$, noting that $l_{1}+l_{t}\leq n-2-d_{G}(v_{2})$, then $d_{G}(v_{2_{1}})\leq l_{1}+l_{t}+4\leq n+2-d_{G}(v_{2})\leq 64$. Then for each $v_{2_{j}}$ ($1\leq j\leq t$), $$\displaystyle \frac{d_{G}(v_{2_{j}})x_{2_{j}}+\sum_{v_{k}\sim v_{2_{j}}} x_{k}}{x_{2_{j}}}\leq d_{G}(v_{2_{j}})+\frac{2+\frac{3(d_{G}(v_{2_{j}})-2)}{n-2}}{\frac{3}{n-2}}  \hspace{4cm}$$
$$\ \ \ = 2d_{G}(v_{2_{j}})-2+\frac{2}{3}(n-2)$$$$\hspace{3cm}\ \  \leq 126 +\frac{2}{3}(n-2)\leq n+2  \hspace{1cm} (n\geq 380).$$
For each $v_{i}\in (N_{G}(v_{1})\backslash \{v_{2_{1}}$, $v_{2_{2}}$, $\ldots$, $v_{2_{t}}\})$, along clockwise direction on $C_{v_{1}}$, $v_{i}$ is between $v_{2_{k}}$ and $v_{2_{k+1}}$ for some $1\leq k \leq t-1$, or $v_{i}$ is between $v_{2_{t}}$ and $v_{2_{1}}$. Then for some $1\leq k\leq t$, $d_{G}(v_{i})\leq l_{k}+2$. Note that $l_{k}\leq n-2-d_{G}(v_{2})$. Then $d_{G}(v_{i})\leq n-d_{G}(v_{2})\leq 62$, and $$\displaystyle \frac{d_{G}(v_{i})x_{i}+\sum_{v_{k}\sim v_{i}} x_{k}}{x_{i}}\leq d_{G}(v_{i})+\frac{1+\frac{3(d_{G}(v_{i})-1)}{n-2}}{\frac{3}{n-2}}  \hspace{3.8cm}$$
$$\ = 2d_{G}(v_{i})-1+\frac{1}{3}(n-2)$$$$\hspace{3cm}\, \leq 123 +\frac{1}{3}(n-2)< n+2  \hspace{1cm} (n\geq 380).$$ For $v_{2}$, $$\displaystyle \frac{d_{G}(v_{2})x_{2}+\sum_{v_{k}\sim v_{2}} x_{k}}{x_{2}}\leq d_{G}(v_{2})+\frac{3d_{G}(v_{2})}{n-2}\leq n+1.\hspace{1cm} (d_{G}(v_{2})\leq n-2)$$

By above discussion, we get $Q(G)X\leq (n+2)X$. By Lemma \ref{le3,3}, we get that $q(G)\leq n+2$.
This completes the proof.
\ \ \ \ \ $\Box$
\end{proof}

\begin{lemma}\label{le3,6} 
Let $G$ be a maximal planar graph with order $n \geq 4$ and $d(v_{1})=\Delta(G) = n -2$. If for $2\leq i\leq n$, $d_{G}(v_{i})< 1+\frac{n}{6}.$ Then $q(G)\leq n-2$.
\end{lemma}

\begin{proof}
 Let $X=(x_{1}$, $x_{2}$, $x_{3}$, $\ldots$, $x_{n})^{T}$ be a positive vector satisfying that $x_{i}$ corresponds to vertex $v_{i}$ and
 $$x_{i}=\left \{\begin{array}{ll}
 1,\ &
 i=1;\\
\\  \frac{4}{n-1},\ & 2\leq i\leq n.\end{array}\right.$$
For $v_{1}$, $$\frac{(n-2)x_{1}+\sum_{v_{j}\sim v_{1}} x_{j}}{x_{1}}\leq n-2+\frac{4(n-2)}{n-1}<n+2.$$
For $v_{i}$ ($2\leq i\leq n$), $$\frac{d_{G}(v_{i})x_{i}+\sum_{v_{j}\sim v_{i}} x_{j}}{x_{i}}\leq d_{G}(v_{i})+\frac{1+\frac{4(d_{G}(v_{i})-1)}{n-1}}{\frac{4}{n-1}}\ \hspace{1cm}$$$$\ \ \hspace{2.1cm}\leq 2d_{G}(v_{i})-1+\frac{n-1}{4}\ $$$$\ \ \hspace{1.3cm}< \frac{7n}{12}+\frac{3}{4}<n+2.$$
As a result, $Q(G)X\leq (n+2)X$. By Lemma \ref{le3,3}, we get that $q(G)\leq n+2$.
This completes the proof.
\ \ \ \ \ $\Box$
\end{proof}

\begin{theorem}\label{th3,7} 
Let $G$ be a maximal planar graph with order $n \geq 380$ and $\Delta(G) = n -2$. Then $q(G)\leq n+2$.
\end{theorem}

\begin{proof}
This theorem follows from Lemmas \ref{le3,4}-\ref{le3,6}.
\ \ \ \ \ $\Box$
\end{proof}

\begin{lemma}\label{le3,8} 
Let $G$ be a maximal planar graph with order $n \geq 91$ and $d_{G}(v_{1})=\Delta(G) = n-1$. There are exactly $1\leq k\leq 13$ vertices $v_{2}$, $v_{3}$, $\ldots$, $v_{k+1}$ in $G$ such that $$\frac{n}{7}+\frac{19}{7}\leq d(v_{i})\leq n-75.$$ For $k+2\leq i\leq n$, $d_{G}(v_{i})<\frac{n}{7}+\frac{19}{7}$. Then $q(G)\leq n+2$.
\end{lemma}

\begin{proof}
Note that in $G$, $d_{G}(v_{i})\geq 3$ for $1\leq i \leq n$, $\displaystyle \sum^{n}_{i=1}d_{G}(v_{i})=2(3n-6)$, and note that if $k\geq 14$, $$k(\frac{n}{7}+\frac{19}{7})+n-1+3(n-k-1)\geq 14(\frac{n}{7}+\frac{19}{7})+n-1+3(n-15)>6n-12.$$ Hence, $k\leq 13$.
Let $X=(x_{1}$, $x_{2}$, $x_{3}$, $\ldots$, $x_{n})^{T}$ be a positive vector satisfying that $x_{i}$ corresponds to vertex $v_{i}$ and
 $$x_{i}=\left \{\begin{array}{ll}
 1,\ &
 i=1;\\
\\ \frac{2}{3k},\ &
 2\leq i\leq k+1;\\
\\ \frac{7}{3(n-k-1)},\ & k+2\leq i\leq n.\end{array}\right.$$
For $v_{1}$, $$\frac{(n-1)x_{1}+\sum_{v_{j}\sim v_{1}} x_{j}}{x_{1}}= n-1+3=n+2.$$
For $v_{i}$ ($k+2\leq i\leq n$), if $d_{G}(v_{i})\geq k+1$, then $$\frac{d_{G}(v_{i})x_{i}+\sum_{v_{j}\sim v_{i}} x_{j}}{x_{i}}\leq d_{G}(v_{i})+\frac{\frac{5}{3}+\frac{7(d_{G}(v_{i})-k-1)}{3(n-k-1)}}{\frac{7}{3(n-k-1)}}\ \ \hspace{4.1cm}$$$$\ \ \hspace{0.3cm}\leq 2d_{G}(v_{i})-k-1+\frac{5}{7}(n-k-1)$$$$\ \ \hspace{1cm}\leq 2d_{G}(v_{i})-\frac{12k}{7}-\frac{12}{7}+\frac{5}{7}n\leq n+2;$$
if $d_{G}(v_{i})\leq k$, then $$\frac{d_{G}(v_{i})x_{i}+\sum_{v_{j}\sim v_{i}} x_{j}}{x_{i}}\leq d_{G}(v_{i})+\frac{\frac{5}{3}}{\frac{7}{3(n-k-1)}}\ \ \hspace{4.1cm}$$$$\ \ \hspace{1.65cm}\leq d_{G}(v_{i})+\frac{5}{7}(n-k-1)< n+2.$$
For $v_{i}$ ($2\leq i\leq k+1$), $d_{G}(v_{i})\geq k$, then $$\frac{d_{G}(v_{i})x_{i}+\sum_{v_{j}\sim v_{i}} x_{j}}{x_{i}}=\frac{(d_{G}(v_{i})-1)x_{i}+x_{i}+\sum_{v_{j}\sim v_{i}} x_{j}}{x_{i}}\ \ \hspace{3.1cm}$$
$$\ \ \hspace{1cm}\leq d_{G}(v_{i})-1+\frac{\sum^{k+1}_{j=1} x_{j}+\frac{7(d_{G}(v_{i})-k)}{3(n-k-1)}}{\frac{2}{3k}}$$$$\leq d_{G}(v_{i})-1+\frac{\frac{5}{3}+\frac{7(d_{G}(v_{i})-k)}{3(n-k-1)}}{\frac{2}{3k}}$$$$\ \ \ \, \hspace{2.9cm}\leq d_{G}(v_{i})-1+\frac{5}{2}k+\frac{\frac{7}{2}k(d_{G}(v_{i})-k)}{n-k-1}.\hspace{1.5cm} (5)$$
As the proof of Lemma \ref{le3,4}, noting that $d_{G}(v_{i})\leq n-75$, we can prove that $(5)\leq n+2.$
Then $Q(G)X\leq (n+2)X$. By Lemma \ref{le3,3}, we get that $q(G)\leq n+2$.
This completes the proof.
\ \ \ \ \ $\Box$
\end{proof}

\begin{lemma}\label{le3,9} 
Let $G$ be a maximal planar graph with order $n \geq 6$ and $d(v_{1})=\Delta(G) = n-1$. If for $2\leq i\leq n$,  $d_{G}(v_{i})<\frac{n}{7}+\frac{19}{7}$. Then $q(G)\leq n+2$.
\end{lemma}

\begin{proof}
Let $X=(x_{1}$, $x_{2}$, $x_{3}$, $\ldots$, $x_{n})^{T}$ be a positive vector satisfying that $x_{i}$ corresponds to vertex $v_{i}$ and
 $$x_{i}=\left \{\begin{array}{ll}
 1,\ &
 i=1;\\
\\  \frac{3}{n-1},\ & 2\leq i\leq n.\end{array}\right.$$
For $v_{1}$, $$\frac{(n-1)x_{1}+\sum_{v_{j}\sim v_{1}} x_{j}}{x_{1}}= n-1+3=n+2.$$
For $v_{i}$ ($2\leq i\leq n$), $$\frac{d_{G}(v_{i})x_{i}+\sum_{v_{j}\sim v_{i}} x_{j}}{x_{i}}\leq d_{G}(v_{i})+\frac{1+\frac{3(d_{G}(v_{i})-1)}{n-1}}{\frac{3}{n-1}}$$$$\ \ \hspace{3cm}=2d_{G}(v_{i})-1+\frac{n-1}{3}$$
$$\ \ \hspace{2.5cm}< \frac{2n}{7}+\frac{31}{7}+\frac{n-1}{3}$$$$\ \ \hspace{2.9cm}\leq n+2. \hspace{1cm} (n\geq 6)$$

As a result, $Q(G)X\leq (n+2)X$. By Lemma \ref{le3,3}, we get that $q(G)\leq n+2$.
This completes the proof.
\ \ \ \ \ $\Box$
\end{proof}

\begin{lemma}\label{le3,10} 
Let $G$ be a maximal planar graph with order $n \geq 461$, $d_{G}(v_{1})=\Delta(G) = n-1$ and $n-81\leq\Delta^{'}(G)\leq n-4$. Then $q(G)\leq n+2$.
\end{lemma}

\begin{proof}
Suppose $d_{G}(v_{2})=\Delta^{'}(G)$. Let $X=(x_{1}$, $x_{2}$, $x_{3}$, $\ldots$, $x_{n})^{T}$ be a positive vector satisfying that $x_{i}$ corresponds to vertex $v_{i}$ and
 $$x_{i}=\left \{\begin{array}{ll}
 1,\ &
 i=1;\\
\\ \frac{4}{7},\ &
i=2;\\
\\ \frac{17}{7(n-2)},\ & 3\leq i\leq n.\end{array}\right.$$
For $v_{1}$, $$\frac{(n-1)x_{1}+\sum_{v_{j}\sim v_{1}} x_{j}}{x_{1}}= n-1+3=n+2.$$
For $v_{2}$, $$\frac{d_{G}(v_{2})x_{2}+\sum_{v_{j}\sim v_{2}} x_{j}}{x_{2}}\leq d_{G}(v_{2})+\frac{1+\frac{17(d_{G}(v_{2})-1)}{7(n-2)}}{\frac{4}{7}}\ \ \hspace{2.1cm}$$$$\ \ \hspace{2.1cm}=d_{G}(v_{2})+\frac{7}{4}+
\frac{17(d_{G}(v_{2})-1)}{4(n-2)}$$$$\ \ \hspace{2cm}<n+2. \hspace{1cm} (d_{G}(v_{2})\leq n-4)$$

Suppose that $v_{1}$ is in the outer face of $G^{\circ}_{v_{1}}$, $C_{v_{1}}=v_{2}v_{3}\cdots v_{n-1}v_{n}v_{2}$ is the Hamiltonian cycle in $G^{\circ}_{v_{1}}$,  $v_{i}s$ ($2\leq i\leq n$) are distributed along clockwise direction on $C_{v_{1}}$, and suppose $N_{G^{\circ}_{v_{1}}}(v_{2})=\{v_{2_{1}}$, $v_{2_{2}}$, $\ldots$, $v_{2_{t}}\}$, where for $1\leq i\leq t-1$, $2_{i}<2_{i+1}$, $v_{2_{1}}=v_{3}$, $v_{2_{t}}=v_{n}$. On $C_{v_{1}}$, along clockwise direction, for $1\leq j\leq t$, suppose that there are $l_{j-1}$ vertices between $v_{2_{j-1}}$ and $v_{2_{j}}$, where if $j=1$, we let $v_{2_{0}}=v_{2}$. Along clockwise direction on $C_{v_{1}}$, suppose that there are $l_{t}$ vertices between $v_{2_{t}}$ and $v_{2}$. For each $v_{2_{j}}$ ($1\leq j\leq t$), noting that $l_{j-1}+l_{j}\leq n-2-d_{G}(v_{2})$ and $d_{G}(v_{2})\geq n-81$, then $d_{G}(v_{2_{j}})\leq l_{j-1}+l_{j}+4\leq n+2-d_{G}(v_{2})\leq 83$, and then $$\ \displaystyle \frac{d_{G}(v_{2_{j}})x_{2_{j}}+\sum_{v_{k}\sim v_{2_{j}}} x_{k}}{x_{2_{j}}}\leq d_{G}(v_{2_{j}})+\frac{\frac{11}{7}+\frac{17(d_{G}(v_{2_{j}})-2)}{7(n-2)}}{\frac{17}{7(n-2)}}  \hspace{3.8cm}$$
$$\ \ \ = 2d_{G}(v_{2_{j}})-2+\frac{11}{17}(n-2)$$$$\hspace{4cm}\  \leq 164 +\frac{11}{17}(n-2)\ \ \hspace{5cm}$$$$\leq n+2.  \hspace{1cm} (n\geq 456)\ \,$$
For each $v_{f}\in (N_{G}(v_{1})\backslash \{v_{2}$, $v_{2_{1}}$, $v_{2_{2}}$, $\ldots$, $v_{2_{t}}\})$, along clockwise direction, then there exists $0\leq s\leq t$ such that $v_{f}$ is between $v_{2_{s}}$ and $v_{2_{s+1}}$ on $C_{v_{1}}$, and then $d_{G}(v_{f})\leq l_{s}+2\leq n-d_{G}(v_{2})\leq 81$. As a result, $$\ \frac{d_{G}(v_{f})x_{f}+\sum_{v_{k}\sim v_{f}} x_{k}}{x_{f}}\leq d_{G}(v_{f})+\frac{\frac{11}{7}+\frac{17(d_{G}(v_{f})-2)}{7(n-2)}}{\frac{17}{7(n-2)}}  \hspace{3.7cm}$$
$$\ \ \ = 2d_{G}(v_{f})-2+\frac{11}{17}(n-2)$$$$\ \, \leq 160 +\frac{11}{17}(n-2)\ \ \hspace{1cm}$$$$\leq n+2.  \hspace{1cm} (n\geq 456)\ $$

As a result, $Q(G)X\leq (n+2)X$. By Lemma \ref{le3,3}, we get that $q(G)\leq n+2$.
This completes the proof.
\ \ \ \ \ $\Box$
\end{proof}

\begin{lemma}\label{le3,11} 
Let $G$ be a maximal planar graph with order $n \geq 15$, $d_{G}(v_{1})=\Delta(G) = n-1$. \\
$(\mathrm{i})$ if $\Delta^{'}(G)= n-2$, then $q(G)< q(\mathcal {H})$;\\
$(\mathrm{ii})$ if $\Delta^{'}(G)= n-3$, then $q(G)< q(\mathcal {H})$.
\end{lemma}

\setlength{\unitlength}{0.6pt}
\begin{center}
\begin{picture}(489,480)
\qbezier(19,386)(19,416)(40,438)\qbezier(40,438)(61,460)(91,460)\qbezier(91,460)(120,460)(141,438)
\qbezier(141,438)(163,416)(163,386)\qbezier(163,386)(163,355)(141,333)\qbezier(141,333)(120,311)(91,311)
\qbezier(91,312)(61,312)(40,333)\qbezier(40,333)(19,355)(19,386)
\put(21,369){\circle*{4}}
\put(86,312){\circle*{4}}
\qbezier(21,369)(53,341)(86,312)
\put(33,430){\circle*{4}}
\qbezier(86,312)(59,371)(33,430)
\put(59,452){\circle*{4}}
\qbezier(86,312)(72,382)(59,452)
\put(126,448){\circle*{4}}
\qbezier(86,312)(106,380)(126,448)
\put(416,314){\circle*{4}}
\put(157,358){\circle*{4}}
\qbezier(86,312)(121,335)(157,358)
\qbezier(59,452)(92,450)(126,448)
\put(133,326){\circle*{4}}
\put(150,387){\circle*{4}}
\put(35,374){\circle*{4}}
\put(98,459){\circle*{4}}
\put(35,397){\circle*{4}}
\put(35,386){\circle*{4}}
\put(146,401){\circle*{4}}
\put(142,415){\circle*{4}}
\put(91,465){$v_{k}$}
\put(25,457){$v_{k-1}$}
\put(81,299){$v_{2}$}
\put(44,329){\circle*{4}}
\put(-2,436){$v_{k-2}$}
\put(1,368){$v_{4}$}
\put(25,321){$v_{3}$}
\put(134,316){$v_{n}$}
\put(161,354){$v_{n-1}$}
\put(124,453){$v_{k+1}$}
\put(78,271){$D_{1}$}
\put(180,-9){Fig. 3.6. $D_{1}-D_{4}$}
\qbezier(341,388)(341,357)(362,335)\qbezier(362,335)(384,314)(415,314)\qbezier(415,314)(445,314)(467,335)
\qbezier(467,335)(489,357)(489,388)\qbezier(341,388)(341,418)(362,440)\qbezier(362,440)(384,462)(415,462)
\qbezier(415,462)(445,462)(467,440)\qbezier(467,440)(489,418)(489,388)
\put(358,340){\circle*{4}}
\qbezier(416,314)(387,327)(358,340)
\put(356,434){\circle*{4}}
\qbezier(416,314)(386,374)(356,434)
\put(479,352){\circle*{4}}
\qbezier(416,314)(447,333)(479,352)
\put(471,434){\circle*{4}}
\qbezier(416,314)(443,374)(471,434)
\qbezier(356,434)(413,434)(471,434)
\put(386,456){\circle*{4}}
\qbezier(471,434)(428,445)(386,456)
\put(449,82){\circle*{4}}
\put(475,376){\circle*{4}}
\put(437,459){\circle*{4}}
\put(356,371){\circle*{4}}
\put(355,386){\circle*{4}}
\put(355,402){\circle*{4}}
\put(474,392){\circle*{4}}
\put(473,406){\circle*{4}}
\put(409,301){$v_{2}$}
\put(339,337){$v_{4}$}
\put(367,460){$v_{k}$}
\put(318,435){$v_{k-1}$}
\put(428,467){$v_{k+1}$}
\put(473,436){$v_{k+2}$}
\put(484,347){$v_{n-1}$}
\put(406,272){$D_{2}$}
\qbezier(19,145)(19,115)(40,94)\qbezier(40,94)(61,73)(91,73)\qbezier(91,73)(120,73)(141,94)
\qbezier(141,94)(163,115)(163,145)\qbezier(19,145)(19,174)(40,195)\qbezier(40,195)(61,217)(91,217)
\qbezier(91,217)(120,217)(141,195)\qbezier(141,195)(163,174)(163,145)
\put(34,131){\circle*{4}}
\put(92,73){\circle*{4}}
\put(30,107){\circle*{4}}
\qbezier(92,73)(61,90)(30,107)
\put(383,320){\circle*{4}}
\put(458,328){\circle*{4}}
\put(369,308){$v_{3}$}
\put(462,320){$v_{n}$}
\put(33,146){\circle*{4}}
\put(130,85){\circle*{4}}
\put(53,84){\circle*{4}}
\put(153,109){\circle*{4}}
\qbezier(92,73)(122,91)(153,109)
\put(92,217){\circle*{4}}
\qbezier(92,73)(92,145)(92,217)
\put(33,187){\circle*{4}}
\qbezier(92,217)(62,202)(33,187)
\put(146,189){\circle*{4}}
\qbezier(92,217)(119,203)(146,189)
\qbezier(92,73)(62,130)(33,187)
\qbezier(92,73)(119,131)(146,189)
\put(127,207){\circle*{4}}
\put(58,209){\circle*{4}}
\put(33,162){\circle*{4}}
\put(149,133){\circle*{4}}
\put(149,147){\circle*{4}}
\put(148,162){\circle*{4}}
\put(86,60){$v_{2}$}
\put(41,72){$v_{3}$}
\put(10,106){$v_{4}$}
\put(86,230){$v_{k}$}
\put(38,222){$v_{k-1}$}
\put(5,196){$v_{k-2}$}
\put(130,220){$v_{k+1}$}
\put(151,193){$v_{k+2}$}
\put(161,110){$v_{n-1}$}
\put(131,79){$v_{n}$}
\put(84,29){$D_{3}$}
\qbezier(344,144)(344,114)(364,93)\qbezier(364,93)(385,73)(415,73)\qbezier(415,73)(444,73)(465,93)
\qbezier(465,93)(486,114)(486,144)\qbezier(344,144)(344,173)(364,194)\qbezier(364,194)(385,215)(415,215)
\qbezier(415,215)(444,215)(465,194)\qbezier(465,194)(486,173)(486,144)
\put(413,73){\circle*{4}}
\put(366,195){\circle*{4}}
\qbezier(413,73)(389,134)(366,195)
\put(344,142){\circle*{4}}
\qbezier(366,195)(355,169)(344,142)
\put(385,208){\circle*{4}}
\qbezier(413,73)(399,141)(385,208)
\put(428,213){\circle*{4}}
\qbezier(413,73)(420,143)(428,213)
\put(453,203){\circle*{4}}
\qbezier(413,73)(433,138)(453,203)
\put(484,159){\circle*{4}}
\qbezier(453,203)(468,181)(484,159)
\put(474,183){\circle*{4}}
\put(350,175){\circle*{4}}
\put(398,202){\circle*{4}}
\put(408,203){\circle*{4}}
\put(419,203){\circle*{4}}
\put(359,100){\circle*{4}}
\qbezier(413,73)(386,87)(359,100)
\put(470,100){\circle*{4}}
\qbezier(413,73)(441,87)(470,100)
\put(369,112){\circle*{4}}
\put(364,124){\circle*{4}}
\put(360,135){\circle*{4}}
\put(468,121){\circle*{4}}
\put(472,134){\circle*{4}}
\put(475,146){\circle*{4}}
\put(385,79){\circle*{4}}
\put(407,60){$v_{2}$}
\put(371,67){$v_{3}$}
\put(338,95){$v_{4}$}
\put(309,142){$v_{k-1}$}
\put(325,183){$v_{k}$}
\put(331,207){$v_{k+1}$}
\put(362,222){$v_{k+2}$}
\put(422,224){$v_{l-2}$}
\put(456,212){$v_{l-1}$}
\put(481,186){$v_{l}$}
\put(489,161){$v_{l+1}$}
\put(475,102){$v_{n-1}$}
\put(451,75){$v_{n}$}
\put(404,37){$D_{4}$}
\end{picture}
\end{center}

\begin{proof}
Suppose $d_{G}(v_{2})=\Delta^{'}(G)$, $v_{1}$ is in the outer face of $G^{\circ}_{v_{1}}$, and suppose that $C_{v_{1}}=v_{2}v_{3}\cdots v_{n-1}v_{n}v_{2}$ is the Hamiltonian
cycle in $G^{\circ}_{v_{1}}$ (see Fig. 3.6).

(i) Suppose $d_{G}(v_{2})=\Delta^{'}(G)= n-2$ and $v_{k}\notin N_{G}(v_{2})$ ($4\leq k\leq n-1$). Then $G^{\circ}_{v_{1}}\cong D_{1}$ (see Fig. 3.6). For convenience, we suppose $G^{\circ}_{v_{1}}= D_{1}$. By Lemma \ref{le2,4}, we get that $q(G)> 15$. Let $X=(x_1, x_2, \ldots, x_{n})^T \in R^{n}$ be the Perron eigenvector corresponding to $q(G)$, where $x_{i}$ corresponds to vertex $v_{i}$.

Note that $$q(G)x_k=3x_k+x_{k-1}+x_{k+1}+x_{1},\ \ \ \hspace{2.1cm} (6)$$
$$\displaystyle q(G)x_2=(n-2)x_2+x_1+\sum_{3\leq i\leq n, i\neq k}x_i. \ \ \hspace{1.5cm} (7)$$ (6), (7) tell us that
$$\displaystyle q(G)x_2-q(G)x_k=(n-2)x_2-3x_k+\sum_{3\leq i\leq k-2} x_i+\sum_{k+2\leq i\leq n} x_i,$$ $$(q(G)-3)(x_2-x_k)=(n-5)x_2+\sum_{3\leq i\leq k-2} x_i+\sum_{k+2\leq i\leq n} x_i.\ \ \ \ $$ Because $n\geq 15$, it follows immediately that $x_2>x_k$.

Note that $$q(G)x_{k-1}=5x_{k-1}+x_{1}+x_2+x_{k-2}+x_k+x_{k+1},$$ $$q(G)x_{k+1}=5x_{k+1}+x_1+x_2+x_{k-1}+x_{k}+x_{k+2}.$$ Then
$$q(G)(x_{k-1}+x_{k+1})=6(x_{k-1}+x_{k+1})+2(x_{1}+x_k+x_2)+x_{k-2}+x_{k+2}. \hspace{1.2cm} (8)$$
From (6) and (7), we also get that $$\displaystyle q(G)(x_2+x_k)=(n-2)x_2+3x_k+2x_1+2x_{k-1}+2x_{k+1}+\sum_{3\leq i\leq k-2} x_i+\sum_{k+2\leq i\leq n} x_i. \hspace{0.6cm} (9)$$
By (9)-(8), we get that $$\displaystyle q(G)(x_2+x_k)-q(G)(x_{k-1}+x_{k+1}) \ \ \ \hspace{8.2cm}$$$$=(n-11)x_2+4(x_{2}+x_k)-4(x_{k-1}+x_{k+1})+3(x_2-x_{k})+\sum_{3\leq i\leq k-3} x_i+\sum_{k+3\leq i\leq n} x_i.$$
It follows that $$\displaystyle (q(G)-4)[x_2+x_k-(x_{k-1}+x_{k+1})]=(n-11)x_2+3(x_2-x_{k})+\sum_{3\leq i\leq k-3} x_i+\sum_{k+3\leq i\leq n} x_i. \hspace{0.5cm} (10)$$
Because $n\geq 15$, (10) tells us that $x_2+x_k>x_{k-1}+x_{k+1}$.

Let $F=G-v_{k-1}v_{k+1}+v_{2}v_{k}$. Note the relation between the Rayleigh quotient and the largest eigenvalue of a non-negative real symmetric matrix, and note that $$X^{T}Q(F)X-X^{T}Q(G)X=(x_2+x_k)^{2}-(x_{k-1}+x_{k+1})^{2}.$$ It follows that when $n\geq 15$, then $q(F)>X^{T}Q(F)X> X^{T}Q(G)X=q(G)$. Because $F\cong \mathcal {H}$,  $q(\mathcal {H})>q(G)$. Then (i) follows.

(ii)  Suppose $d_{G}(v_{2})=\Delta^{'}(G)= n-3$. Then there are three cases for $G$, that is, $G\cong D_{2}$, $G\cong D_{3}$ or $G\cong D_{4}$ (see Fig. 3.6). By Lemma \ref{le2,4}, we know that $q(G) > 15$.

{\bf Case 1} $G\cong D_{2}$. For convenience, we suppose that $G= D_{2}$. Because $d_{G}(v_{2})=\Delta^{'}(G)= n-3$, $4\leq k\leq n-2$. Let $X=(x_1, x_2, \ldots, x_{n})^T \in R^{n}$ be the Perron eigenvector corresponding to $q(G)$, where $x_{i}$ corresponds to vertex $v_{i}$.

Note that $$q(G)x_{k+1}=3x_{k+1}+x_1+x_{k}+x_{k+2}, \ \ \hspace{3cm} (11)$$
$$q(G)x_k=4x_k+x_1+x_{k-1}+x_{k+1}+x_{k+2}.  \hspace{2.5cm} (12)$$
By (12)-(11), we get $$q(G)(x_{k}-x_{k+1})=3x_{k}-2x_{k+1}+x_{k-1}. \hspace{2.8cm}$$ Then $$(q(G)-2)(x_{k}-x_{k+1})=x_{k}+x_{k-1}.\ \ \ \hspace{2.5cm} (13)$$
Because $n\geq 15$, (13) implies $x_{k} > x_{k+1}$.

Note that $$\ q(G)x_2=(n-3)x_2+x_1+\sum_{3\leq i\leq k-1} x_i+\sum_{k+2\leq i\leq n} x_i. \ \ \hspace{1.5cm} (14)$$
By (14)+(12), we get
$$q(G)(x_2+x_k)=(n-3)x_2+2x_1+4x_k+2x_{k-1}+x_{k+1}+2x_{k+2}+\sum_{3\leq i\leq k-2}x_{i}+\sum_{k+3\leq i\leq n}x_{i}. \hspace{0.7cm} (15)$$  Note that $$q(G)x_{k-1}=5x_{k-1}+x_{1}+x_2+x_{k-2}+x_k+x_{k+2}.\ \ \ \  \hspace{2cm} (16)$$
By (16)+(11), we get that $$q(G)(x_{k-1}+x_{k+1})=5x_{k-1}+x_{k-2}+2x_{1}+x_{2}+2x_k+3x_{k+1}+2x_{k+2}.\ \ \hspace{1cm} (17)$$
By (17)-(12), we get that $$q(G)(x_{k-1}+x_{k+1}-x_k)=x_{1}+x_{2}+x_{k-2}+4x_{k-1}+2x_{k+1}+x_{k+2}-2x_k.$$
Then $$(q(G)-2)(x_{k-1}+x_{k+1}-x_k)=x_{1}+x_2+x_{k-2}+2x_{k-1}+x_{k+2}>0. \hspace{1cm} (18)$$
Because $n\geq 15$, (18) implies that $x_{k-1}+x_{k+1}>x_k$.

By (14)-(11), we get that
$$q(G)(x_2-x_{k+1})=(n-3)x_2+x_{k-1}-x_k-3x_{k+1}+\sum_{3\leq i\leq k-2}x_{i}+\sum_{k+3\leq i\leq n}x_{i}>0.$$
Then $$(q(G)-4)(x_2-x_{k+1})=(n-7)x_2+x_{k-1}+x_{k+1}-x_k+\sum_{3\leq i\leq k-2}x_{i}+\sum_{k+3\leq i\leq n}x_{i}>0. \hspace{0.7cm} (19)$$
Because $n\geq 15$, (19) implies $x_2>x_{k+1}$.

By (14)-(12), we get that
$$q(G)(x_2-x_k)=(n-3)x_2-4x_k-x_{k+1}+\sum_{3\leq i\leq k-2}x_{i}+\sum_{k+3\leq i\leq n}x_{i}.$$
Then $$(q(G)-4)(x_2-x_k)=(n-8)x_2+x_2-x_{k+1}+\sum_{3\leq i\leq k-2}x_{i}+\sum_{k+3\leq i\leq n}x_{i}>0. \hspace{1cm} (20)$$
Because $n\geq 15$, (20) implies that $x_2>x_k$.

By (14)-(16), we get that
$$q(G)(x_2-x_{k-1})=(n-4)x_2-4x_{k-1}-x_k+\sum_{3\leq i\leq k-3}x_{i}+\sum_{k+3\leq i\leq n}x_{i}.$$
Then $$(q(G)-4)(x_2-x_{k-1})=(n-9)x_2+x_{2}-x_k+\sum_{3\leq i\leq k-3}x_{i}+\sum_{k+3\leq i\leq n}x_{i}. \hspace{0.7cm} (21)$$
Because $n\geq 15$, (21) implies that $x_2>x_{k-1}$.

Note that
$$q(G)x_{k+2}=6x_{k+2}+x_{k+1}+x_{k}+x_{k-1}+x_{k+3}+x_2+x_1.\ \ \hspace{2.2cm} (22)$$
By (14)-(22), we get that
$$q(G)(x_2-x_{k+2})=(n-4)x_2-5x_{k+2}-x_k-x_{k+1}+\sum_{3\leq i\leq k-2}x_{i}+\sum_{k+4\leq i\leq n}x_{i}.$$
Then $$(q(G)-5)(x_2-x_{k+2})=(n-11)x_2+x_{2}-x_k+x_{2}-x_{k+1}+\sum_{3\leq i\leq k-2}x_{i}+\sum_{k+4\leq i\leq n}x_{i}. \hspace{0.5cm} (23)$$ Because $n\geq 15$, (23) implies that $x_2>x_{k+2}$.

By (16)+(22), we get that
$$q(G)(x_{k-1}+x_{k+2})=2x_1+2x_2+x_{k-2}+6x_{k-1}+2x_{k}+x_{k+1}+7x_{k+2}+x_{k+3}. \hspace{1cm} (24)$$
By (15)-(24), we get that $$q(G)(x_{2}+x_k)-q(G)(x_{k-1}+x_{k+2}) \hspace{8.8cm}$$
$$=(n-5)x_2-4x_{k-1}+2x_{k}-5x_{k+2}+\sum_{3\leq i\leq k-3}x_{i}+\sum_{k+4\leq i\leq n}x_{i} \ \ \hspace{4cm}$$
$$=(n-14)x_2+4x_2-4x_{k-1}+2x_{k}+5x_2-5x_{k+2}+\sum_{3\leq i\leq k-3}x_{i}+\sum_{k+4\leq i\leq n}x_{i}. \hspace{1cm} (25)$$
Because $n\geq 15$, (25) implies that $x_{2}+x_k>x_{k-1}+x_{k+2}$.

Let $\mathbb{F} = G-v_{k-1}v_{k+2}+v_{2}v_{k}$. Note that $X^{T}Q(\mathbb{F})X-X^{T}Q(G)X = (x_{2}+x_k)^{2}-(x_{k-1}+x_{k+2})^{2}$.
It follows that when $n\geq 15$, then $q(\mathbb{F}) > X^{T}Q(\mathbb{F})X > X^{T}Q(G)X = q(G)$. By (i), it
follows immediately that $q(\mathcal {H}) > q(\mathbb{F}) > q(G)$.

{\bf Case 2} $G\cong D_{3}$.

For convenience, we suppose that $G= D_{3}$. Because $d_{G}(v_{2})=\Delta^{'}(G)= n -3$, $5\leq k\leq n-2$. Let $X=(x_1$, $x_2$, $\ldots$, $x_{n})^T \in R^{n}$ be the Perron eigenvector corresponding to $q(G)$, where $x_{i}$ corresponds to vertex $v_{i}$.

Note that $$q(G)x_{k+1}=3x_{k+1}+x_{1}+x_k+x_{k+2}, \hspace{4cm} (26)$$
$$q(G)x_{k-1}=3x_{k-1}+x_1+x_{k}+x_{k-2}, \ \ \ \hspace{3.7cm} (27)$$
$$q(G)x_2=(n-3)x_2+x_1+x_k+\sum_{3\leq i\leq k-2}x_{i}+\sum_{k+2\leq i\leq n}x_{i}.\ \ \ \ \hspace{0.3cm} (28)$$
Then $$q(G)x_2-q(G)x_{k+1}=(n-3)x_2-3x_{k+1}+\sum_{3\leq i\leq k-2}x_{i}+\sum_{k+3\leq i\leq n}x_{i},$$
$$(q(G)-3)(x_2-x_{k+1})=(n-6)x_2+\sum_{3\leq i\leq k-2}x_{i}+\sum_{k+3\leq i\leq n}x_{i}>0.$$
This implies that $x_2>x_{k+1}$.
By (28)-(27), we get that $$q(G)x_2-q(G)x_{k-1}=(n-3)x_2-3x_{k-1}+\sum_{3\leq i\leq k-3}x_{i}+\sum_{k+2\leq i\leq n}x_{i}.$$
Then $$(q(G)-3)(x_2-x_{k-1})=(n-6)x_2+\sum_{3\leq i\leq k-3}x_{i}+\sum_{k+2\leq i\leq n}x_{i}>0.$$
This implies that $x_2>x_{k-1}$.
Note that $$q(G)x_k=6x_k+x_1+x_2+x_{k-2}+x_{k-1}+x_{k+1}+x_{k+2},\ \   \hspace{1.2cm} (29)$$
$$q(G)x_{k+2}=5x_{k+2}+x_1+x_2+x_{k+1}+x_{k}+x_{k+3}.\ \ \ \  \hspace{2cm} (30)$$
By (28)-(29), we get $$q(G)x_2-q(G)x_k=(n-4)x_2-5x_k-x_{k-1}-x_{k+1}+\sum_{3\leq i\leq k-3}x_{i}+\sum_{k+3\leq i\leq n}x_{i}.$$
Then $$(q(G)-5)(x_2-x_k)=(n-11)x_2+2x_2-x_{k-1}-x_{k+1}+\sum_{3\leq i\leq k-3}x_{i}+\sum_{k+3\leq i\leq n}x_{i}.$$
This implies that $x_2>x_k$.
By (28)-(30), we get $$q(G)x_2-q(G)x_{k+2}=(n-4)x_2-4x_{k+2}-x_{k+1}+\sum_{3\leq i\leq k-2}x_{i}+\sum_{k+4\leq i\leq n}x_{i}.$$
Then $$(q(G)-4)(x_2-x_{k+2})=(n-9)x_2+x_2-x_{k+1}+\sum_{3\leq i\leq k-2}x_{i}+\sum_{k+4\leq i\leq n}x_{i}.$$
This implies that $x_2>x_{k+2}$.
By (26)+(28), we get that
$$q(G)(x_2+x_{k+1})=(n-3)x_2+3x_{k+1}+2x_1+2x_{k}+2x_{k+2}+\sum_{3\leq i\leq k-2}x_{i}+\sum_{k+3\leq i\leq n}x_{i}.\  \hspace{0.3cm} (31)$$
By (29)+(30), we get that
$$q(G)(x_k+x_{k+2})=2x_1+2x_2+2x_{k+1}+7x_{k}+6x_{k+2}+x_{k-1}+x_{k-2}+x_{k+3}.\   \hspace{1.5cm} (32)$$
By (31)-(32), we get that
$$q(G)(x_2+x_{k+1})-q(G)(x_k+x_{k+2}) \hspace{10cm}$$
$$=(n-5)x_2+x_{k+1}-5x_{k}-4x_{k+2}-x_{k-1}+\sum_{3\leq i\leq k-3}x_{i}+\sum_{k+4\leq i\leq n}x_{i} \ \ \hspace{5cm}$$
$$=(n-15)x_2+x_{k+1}+5x_2-5x_{k}+4x_2-4x_{k+2}+x_2-x_{k-1}+\sum_{3\leq i\leq k-3}x_{i}+\sum_{k+4\leq i\leq n}x_{i}. \hspace{1cm}(33)$$
Because $n\geq 15$, (33) implies that $x_2+x_{k+1}>x_k+x_{k+2}$.

Let $\mathbb{F}=G-v_{k}v_{k+2}+v_{2}v_{k+1}$. Note that $X^{T}Q(\mathbb{F})X-X^{T}Q(G)X=(x_2+x_{k+1})^{2}-(x_k+x_{k+2})^{2}$. It follows that when $n\geq 15$, then $q(\mathbb{F})>X^{T}Q(\mathbb{F})X> X^{T}Q(G)X=q(G)$. Because $n\geq 15$, by (i), it follows immediately that $q(\mathcal {H})>q(\mathbb{F})>q(G)$.

{\bf Case 3} $G\cong D_{4}$.

For convenience, we suppose that $G= D_{4}$. Because $d_{G}(v_{2})=\Delta^{'}(G) = n -3$, $4\leq k\leq l-2$, $l\leq n-1$. Let $X=(x_1$, $x_2$, $\ldots$, $x_{n})^T \in R^{n}$ be the Perron eigenvector corresponding to $q(G)$, where $x_{i}$ corresponds to vertex $v_{i}$. Let $\mathbb{F}=G-v_{l-1}v_{l+1}+v_{2}v_{l}$.  As (i), it can be proved that $q(G)<q(\mathbb{F})$. By (i), we get that $q(\mathbb{F})<q(\mathcal {H})$. Then $q(G)<q(\mathcal {H})$.

From above three cases, (ii) follows. This completes the proof.
\ \ \ \ \ $\Box$
\end{proof}

\begin{theorem}\label{th3,12} 
Let $G$ be a maximal planar graph with order $n \geq 456$, $d(v_{1})=\Delta(G) = n -1$ and $\Delta^{'}(G)\leq n-2$. Then $q(G)<q(\mathcal {H})$.
\end{theorem}

\begin{proof}
This theorem follows from Lemmas \ref{le3,2} and \ref{le3,8}-\ref{le3,11}.
\ \ \ \ \ $\Box$
\end{proof}

\begin{theorem}\label{th3,13} 
Let $G$ be a maximal planar graph with order $n \geq 456$. Then $q(G)\leq q(\mathcal {H})$ with equality if and only if $G\cong\mathcal {H}$.
\end{theorem}

\begin{proof}
This theorem follows from Lemmas \ref{le3,1}, \ref{le3,2}, Theorems \ref{th3,7} and \ref{th3,12}.
\ \ \ \ \ $\Box$
\end{proof}

\begin{theorem}\label{th3,14} 
Let $G$ be a planar graph with order $n \geq 456$. Then $q(G)\leq q(\mathcal {H})$ with equality if and only if $G\cong\mathcal {H}$.
\end{theorem}

\begin{proof}
This theorem follows from the narrations in Section 2 and Theorem \ref{th3,13}.
\ \ \ \ \ $\Box$
\end{proof}

{\bf Remark} By computations with computer, we can check that among all planar graphs of order $n \leq 6$, the graph $\mathcal {H}$ has the maximal signless Laplacian spectral radius. And by computation with computer, we can check and find which graph has the maximal signless Laplacian spectral radius among all planar graphs with not large order. But it seems obviously it is not a good way by computation to check and find which graph has the maximal signless Laplacian spectral radius among all planar graphs with large order because magnitudes of work need do. As for the planar graphs of order $n \leq 455$, we think that it need a smart way to check and find which graph has the maximal signless Laplacian spectral radius. By some computations with computer, we conjuncture that among all planar graphs of order $n \leq 455$, the graph $\mathcal {H}$ still has the maximal signless Laplacian spectral radius.

\small {

}

\end{document}